\documentclass[]{interact}

\usepackage{epstopdf}
\usepackage[caption=false]{subfig}

\usepackage[numbers,sort&compress]{natbib}
\bibpunct[, ]{[}{]}{,}{n}{,}{,}
\makeatletter
\def\NAT@def@citea{\def\@citea{\NAT@separator}}
\makeatother

\theoremstyle{plain}
\newtheorem{theorem}{Theorem}[section]
\newtheorem{lemma}{Lemma}[section]

\newtheorem{proposition}{Proposition}[section]

\theoremstyle{definition}
\newtheorem{definition}{Definition}[section]
\newtheorem{example}{Example}[section]
\newtheorem{question}{Question}[section]

\theoremstyle{remark}
\newtheorem{remark}{Remark}[section]

\usepackage{amsmath,amsfonts,amssymb,array,mathdots}
\usepackage{graphicx}

\usepackage{chngcntr}
\counterwithin{equation}{section}

\usepackage{pgfplots}
\usepackage{graphicx}

\DeclareMathOperator{\Sol}{Sol}

\DeclareMathOperator{\Q}{Q}

\DeclareMathOperator{\Po}{P}

\DeclareMathOperator{\VI}{VI}
\DeclareMathOperator{\ri}{ri}
\DeclareMathOperator{\cl}{cl}
\DeclareMathOperator{\N}{\mathbb{N}}
\DeclareMathOperator{\R}{\mathbb{R}}

\DeclareMathOperator{\A}{\mathcal{A}}
\DeclareMathOperator{\B}{\mathcal{B}}
\DeclareMathOperator{\Ba}{\mathbb{B}}

\DeclareMathOperator{\VVI}{VVI}

\DeclareMathOperator{\inte}{int}

\begin{document}


\title{Disconnectedness and unboundedness of the solution sets of monotone vector variational inequalities}

\author{
\name{Vu Trung Hieu\textsuperscript{a}\thanks{CONTACT Vu Trung Hieu. Email: hieuvut@gmail.com}\thanks{Dedicated to my colleagues at Phuong Dong University, Associate Professor Phan Huu Huan and Associate Professor Nguyen Kim Vu, on the occasion of their retirement}}
\affil{\textsuperscript{a} Division of Mathematics, Phuong Dong University, Hanoi, Vietnam}
}

\maketitle

\begin{abstract} In this paper, we investigate the topological structure of solution sets of monotone vector variational inequalities. We show that if the weak Pareto solution set of a monotone vector variational inequality is disconnected, then each connected component of the set is unbounded. Similarly, this property holds for the proper Pareto solution set. Two open questions on the topological structure of the solution sets of (symmetric) monotone vector variational inequalities are raised at the end of the paper.
\end{abstract}

\begin{keywords}
Monotone vector variational inequality; solution set;  disconnectedness; unboundedness; scalarization formula
\end{keywords}

\section{Introduction}
The concept of vector variational inequality (VVI for short) was introduced by Giannessi in his well-known paper \cite{G80}. The class of monotone VVIs appeared in the research on convex vector optimization problems and linear fractional vector optimization
problems.
The topological structure of the solution sets of monotone VVIs was studied in \cite{Hieu19,LKLY98,LY2000,YP2000,YY2011}.

In \cite{LY2000}, the authors have proved that the weak Pareto solution set of a monotone VVI is connected if the constraint set is compact. By using a result on the solution stability of Robinson \cite{Robinson79} for monotone affine variational inequalities and a scalarization method, Yen and Yao \cite{YY2011} have shown that if the weak Pareto solution set of a monotone affine VVI is disconnected then each connected component of this set is unbounded. So, the set is connected if it is bounded and nonempty.  Similar assertions are valid for the Pareto solution set.

The present paper can be considered as a new attempt to develop and extend the results of Yen and Yao for general monotone VVIs. Based on scalarization formulae of VVIs \cite{LKLY98} and results of Facchinei and Pang \cite{FaPa03} on the solution stability of monotone variational inequalities, we prove that if the weak Pareto solution set of a monotone VVI is disconnected, then each connected component of this set is unbounded. The proper Pareto solution set of a monotone VVI has the same property. Consequently, when the constraint set is polyhedral convex, this property holds for the Pareto solution set. An open question on the property for the Pareto solution set of a monotone VVI is raised.

We have known that the weak Pareto solution set (the Pareto solution set, and the proper Pareto solution set) of a non-symmetric monotone VVI (see, e.g., Example \ref{example2}) could be disconnected.
Up to now, we have never known a symmetric monotone VVI whose weak Pareto solution set (or, Pareto solution set, proper Pareto solution set) is disconnected in the literature. The topological structure of the solution sets of non-symmetric monotone VVIs is more interesting than that of symmetric monotone VVIs.  At the end of the paper, we give another open question on the connectedness of the solution sets of symmetric monotone VVIs.

The remaining part of this paper consists of four sections. Section \ref{sec_pre} gives some definitions, notations, and auxiliary results on VVIs. Section \ref{sec_basic} establishes some technical results concerning the basic multifunctions of monotone VVIs. The main results on the disconnectedness and the unboundedness of the solution sets of monotone VVIs are shown in Section~\ref{sec_dis}. The last section gives two examples and presents two open questions.

\section{Preliminaries}\label{sec_pre}

The scalar product of $x$ and $y$ from $\R^n$ is denoted by $\langle x,y\rangle$. Let $K\subset\R^n$ be a nonempty closed convex set and  $F_l: K\to\mathbb R^n$ ($l=1,\dots, m$) be continuous vector-valued functions. We denote $F=(F_1,\dots, F_m)$ and
\begin{center}
	$F(x)(u)=( \langle F_1(x),u\rangle, \dots, \langle F_m(x),u\rangle), \ \forall x\in K,\ \forall u\in\R^n$.
\end{center}
Let $C=\R^m_+$ and
$$\Delta=\Big\{(\xi_1,\dots,\xi_m)\in \R^m_+\,:\,
\sum_{l=1}^m\xi_l=1\Big\},$$
where $\R^m_+$ is the nonnegative orthant of  $\R^m$. The relative interior of $\Delta$ is described by the formula $${\ri}\Delta=\{\xi\in\Delta\,:\, \xi_l>0,\ l=1,\dots,m\}.$$

\begin{definition}{\rm (see \cite{G80})}
	The problem
	$${ \VVI(F,K)} \qquad {\text{Find}}\ \, x\in K \ \; \text{such that}\ \, F(x)(y-x)\nleq_{C\setminus\{0\}}0,\ \, \forall y\in K,$$
	is said to be the \textit{vector variational inequality} defined by $F,K$ and $C$. The inequality means that
	$F(x)(x-y)\notin C\setminus\{0\}$.
	The solution set of $\VVI(F,K)$ is denoted by $\Sol(F,K)$ and called the \textit{Pareto solution set}.
\end{definition}

\begin{definition}{\rm (see \cite{ChYa90})}
	The problem
	$${\VVI^w(F,K)} \qquad {\text{Find}}\ \, x\in K \ \; \text{such that}\ \, F(x)(y-x)\nleq_{\inte C}0, \ \, \forall y\in K,$$
	where ${\inte}C$ is the interior of $C$ and the inequality means $F(x)(x-y)\notin {\inte}C$, is called the \textit{weak vector variational inequality} defined by $F,K$ and $C$. The solution set of $\VVI^w(F,K)$ is denoted by $\Sol^w(F,K)$ and called the \textit{weak Pareto solution set} of the problem $\VVI(F,K)$.
\end{definition}

For $m=1$, both $\VVI(F,K)$ and $\VVI^w(F,K)$ coincide with the
classical variational inequality problem
$$\VI(F,K) \qquad {\text{Find}}\ \, x\in K \ \; \text{such that}\ \, \langle F(x),y-x\rangle\geq 0,\ \, \forall y\in K.$$
Note that $x$ solves $\VI(F,K)$ if and only if
$F(x)\in -N_K(x),$
where $N_K(x)$ is the normal cone  of $K$ at $x\in K$,  which is defined by
$$
N_K(x)=\{x^*\in \R^n:\langle x^*,y-x\rangle\leq 0, \forall y\in K\}.
$$
Clearly, when $K=\R^n$, $x$ solves $\VI(F,K)$  if and only if $x$ is a zero point of the function $F$.

The solution sets of  $\VVI(F,K)$ can be computed  via certain unions of the solution sets of the parametric variational inequality $\VI(F_{\xi},K)$, where
$$F_{\xi}(x)= \sum_{l=1}^m\xi_lF_l(x),$$
 with $\xi\in \Delta$.

\begin{theorem} {\rm (see \cite{LKLY98}, \cite{LY01})}\label{thm_scalarization} It holds that
	\begin{equation}\label{scalarization}\bigcup_{\xi\in{\ri}\Delta}{\Sol(F_{\xi},K)}\subset {\Sol(F,K)}\subset {\Sol^w}{(F,K)}=\bigcup_{\xi\in\Delta}\Sol(F_{\xi},K). \end{equation}
	If $K$ is a polyhedral convex set, i.e.,
	$K$ is the intersection of finitely many closed half-spaces of
	$\mathbb R^n$, then the first inclusion in \eqref{scalarization}
	holds as equality.
\end{theorem}

If  $x\in \Sol(F_{\xi},K)$ for some $\xi\in\ri\Delta$, then $x$ is said to be a {\it proper Pareto solution} \cite{HYY2015b} of $\VVI(F,K)$. Here, $\Sol^{pr}(F,K)$ stands for the proper Pareto solution set of $\VVI(F,K)$. With this notion, we have
\begin{equation}\label{Sol_pr}
\Sol^{pr}(F,K)=\bigcup_{\xi\in{\ri}\Delta}{\Sol(F_{\xi},K)}.
\end{equation}

\begin{definition} {\rm (see \cite{LY2000})}
	One says that the problem $\VVI(F,K)$ is monotone if all the functions $F_l$, where $l = 1,\dots, m$, are monotone,  i.e., $$\langle  F_l(y)-F_l(x),y-x\rangle \geq 0$$ for all $x,y\in K$.
\end{definition}

Remind that if the problem $\VVI(F,K)$ is monotone, then the parametric variational inequality $\VI(F_{\xi},K)$ is monotone for any $\xi\in \Delta$. In this case, $\Sol(F_{\xi},K)$ is a convex set.

Let $X$ be a subset of $\R^n$. Recall that $X$ is \textit{connected} if there does not exist two nonempty disjoint subsets $A,B$ of $X$ and two open subsets $U,V$ in $\R^n$ such that $A\subset U$, $B\subset V$ and $U\cap V=\emptyset$.
A nonempty subset $A\subset X$ is said to be a \textit{connected component} of  $X$ if $A$ is connected and it is not a proper subset of any connected subset of $X$. Remind that when $X$ is connected, if the set $A\subset X$ is closed and open in $X$ then $A=X$. The closure $\cl(X)$ is connected when $X$ is connected.

\section{Basic multifunctions}\label{sec_basic}

The \textit{basic multifunction} associated to the problem $\VVI(F,K)$ is defined by
$$S:\Delta\rightrightarrows \R^n, \ S(\xi)=	\Sol(F_{\xi},K).$$
Theorem \ref{thm_scalarization} yields $\Sol^{w}(F,K)=S(\Delta)$ and $\Sol^{pr}(F,K)=S(\ri\Delta)$. Hence, one can use the basic multifunction $S$ to investigate different properties of the solution sets of $\VVI(F,K)$. The inverse multifunction $S^{-1}: \R^n \rightrightarrows  \Delta $ is defined by
$$S^{-1}(x)=\{\xi\in \Delta: x\in S(\xi)\}, \ x\in \R^n.$$
If $\A$ is a subset in $\R^n$, then the inverse image of $\A$ by $S$ is the following set
\begin{equation}\label{S_1}
S^{-1}(\A)=\{\xi\in\Delta:S(\xi)\cap \A\neq \emptyset\}.
\end{equation}

\begin{proposition}\label{A_compact}
	Let $\A$ be a subset of the weak Pareto solution set. If $\A$ is compact, then $S^{-1}(\A)$ is closed.
\end{proposition}
\begin{proof} Suppose that $\A$ is compact.	Let $\{\xi^k\}\subset S^{-1}(\A)$ be a convergent sequence and $\xi^k\to\bar\xi$. There exists a sequence $\{x^k\}\subset \A$ such that $x^k\in S(\xi^k)$ for every $k\in\N$. By the compactness of $\A$, without loss of generality, we can assume that  $x^{k} \to \bar x$ and $\bar x\in \A$. By definition, one has	$$\left\langle  F_{\xi^k}(x^{k}),y-x^{k}\right\rangle\geq 0, \  \forall y\in K.$$
From the last result, for any $y\in K$,	letting $k\to \infty$, we have $F_{\xi^k}(x^k)\to F_{\bar \xi}(\bar x)$ and
	$$\left\langle F_{\bar\xi}(\bar x),y-\bar x\right\rangle  \geq 0.$$ This means that $\bar x\in  S(\bar\xi)$. We thus get $\bar x\in  \A\cap S(\bar\xi)\neq\emptyset$, and then $\bar\xi\in S^{-1}(\A)$. Hence, the inverse image $S^{-1}(\A)$ is closed.
\end{proof}

\begin{proposition}\label{AB_disjoint} Assume that $\VVI(F,K)$ is monotone. If $\A$ and $\B$ are different connected components of the weak Pareto solution set (or the proper Pareto solution set), then  \begin{equation}\label{empty}
	S^{-1}(\A)\cap S^{-1}(\B)=\emptyset.
\end{equation}
\end{proposition}
\begin{proof}
	Suppose that $\A$ and $\B$ are different connected components of $\Sol^w(F,K)$. On the contrary, there exists
	$\xi\in S^{-1}(\A)\cap S^{-1}(\B).$
It follows that $S(\xi)\cap \A\neq \emptyset$ and $S(\xi)\cap \B\neq \emptyset$. The monotonicity of $F$ implies the connectedness of $S(\xi)$. Since $\A,\B$ are connected components, one has $S(\xi)\subset \A$ and $S(\xi)\subset \B$. This yields $$S(\xi)\subset \A\cap\B\neq\emptyset,$$ which contradicts to the assumption that $\A,\B$ are different. Thus \eqref{empty} has been proved.

The proof for the proper Pareto solution set is similar.
\end{proof}

Let $Q:K\to\R^n$ be a given continuous function, and $\varepsilon$ be a given positive number. We denote by $\Ba(Q,\varepsilon,K)$ the set of all continuous functions $G$ satisfying
$$\sup_{x\in K}\|G(x)-Q(x)\|< \varepsilon.$$

Based on the results \cite[Proposition 5.5.3]{FaPa03} and \cite[Theorem 5.5.15]{FaPa03} on the solution stability of monotone variational inequalities, we obtain the following lemma which will be used in the proof of Proposition \ref{A_open}.

\begin{lemma}\label{FaPa}
	Assume that $\VI(Q,K)$ is monotone and $\Sol(Q,K)$ is nonempty and bounded.  For any open set $U$ containing $\Sol(Q,K)$, there exists $\varepsilon>0$ such
	that $\Sol(G,K)\cap U$ is nonempty for every function $G$ satisfying
\begin{equation}\label{G}
	G\in \Ba(Q,\varepsilon,K\cap \cl(U)).
\end{equation}
\end{lemma}
\begin{proof} Let $U$ be an open set and $U\supset\Sol(Q,K)$. Since the last one is nonempty and bounded, according to \cite[Theorem 5.5.15]{FaPa03},  there exists $\delta>0$ such
	that $\Sol(G,K)\cap U$ is nonempty for every continuous function $G$ satisfying
	$$\|G(x)-Q(x)\|<\delta(1+\|x\|), \ \forall x\in K\cap \cl(U).$$
Therefore, from \cite[Proposition 5.5.3]{FaPa03}, there exists $\varepsilon>0$ such
	that $\Sol(G,K)\cap U$ is nonempty for every function $G$ satisfying \eqref{G}. The proof is complete.
\end{proof}

\begin{proposition}\label{A_open} Assume that $\VVI(F,K)$ is monotone. The following assertions are valid:
	\begin{description}
		\item[\rm (a)] If $\A$ is a bounded connected component of the weak Pareto solution set,  then $S^{-1}(\A)$ is open in $\Delta$.
		\item[\rm (b)] If $\A$ is a bounded connected component of the proper Pareto solution set,  then $S^{-1}(\A)$ is open in $\ri\Delta$.
	\end{description}
\end{proposition}
\begin{proof} $\rm(a)$ Suppose that $\A$ is a bounded connected component of $\Sol^w(F,K)$.
There exists an open set $U$ in $\R^n$ such that $U$ is bounded and $\A \subset U$.

	Let $\eta$ be a fixed point in $S^{-1}(\A)$. Clearly, $\Sol(F_\eta,K)$ is a nonempty and bounded subset of $\A$.
	By the monotonicity of $F_{\eta}$,  for the open set $U$ that  contains $\Sol(F_\eta,K)$, according to Lemma \ref{FaPa}, there exists $\varepsilon >0$ such that
	\begin{equation}\label{Sol_nonempty}
	\Sol(G,K)\cap U\neq \emptyset, \; \ \forall G \in \Ba(F_{\eta},\varepsilon,K\cap\cl(U)).
	\end{equation}
	Taking
	\begin{equation}\label{O}
	O_{\eta}=\left\lbrace \xi\in \Delta: F_{\xi} \in \Ba(F_{\eta},\varepsilon,K\cap\cl(U))\right\rbrace,
	\end{equation}
	we conclude that $O_{\eta}$ is nonempty and open in $\Delta$. Indeed, we have $\eta\in O_{\eta}$ and thus $O_{\eta}\neq \emptyset$.	Let $\{\xi^k\}$ be convergent sequence in the complement $\Delta\setminus O_{\eta}$ and $\xi^k\to\bar \xi$. One has
	$$\sup_{x\in K\cap \cl(U)}\|F_{\xi^k}(x)-F_{\eta}(x)\|\geq \varepsilon.$$
	Taking $k\to+\infty$,
	it follows that $F_{\xi^k}\to F_{\bar\xi}$ on $K\cap\cl(U)$ and
	$$\sup_{x\in K\cap \cl(U)}\|F_{\bar\xi}(x)-F_{\eta}(x)\|\geq \varepsilon.$$
We thus get $\bar \xi\in \Delta\setminus O_{\eta}$. From this, $O_{\eta}$ is open in $\Delta$.

	We will show that $O_{\eta}\subset S^{-1}(\A)$. For any $\xi\in O_{\eta}$, from \eqref{O} and \eqref{Sol_nonempty}, one has
	\begin{equation}\label{SolFxi}
	F_{\xi} \in \Ba(F_{\eta},\varepsilon,K\cap\cl(U)), \ \; \Sol(F_{\xi},K)\cap U\neq \emptyset.
	\end{equation}
We claim that	 the nonempty convex set $\Sol(F_{\xi},K)$ does not intersect with $\Sol^w(F,K)\setminus\A$. Indeed, on the contrary, if it were true that
	$$\Sol(F_{\xi},K)\cap\left(\Sol^w(F,K)\setminus\A\right) \neq\emptyset,$$
we would have
	$$\Sol(F_{\xi},K)\subset\left(\Sol^w(F,K)\setminus\A\right) .$$
	This leads to $\Sol(F_{\xi},K)\cap U=\emptyset$, contrary to \eqref{SolFxi}. Hence, we get $\Sol(F_{\xi},K)\subset \A$. Then the desired conclusion is proved.

	Remind that $\eta$ is an arbitrary point taken in $S^{-1}(\A)$. From what has already been proved, $S^{-1}(\A)$ is open in the
	induced topology of $\Delta$.

	$\rm(b)$ We apply the argument in the proof of $\rm(a)$ again, with $\Delta$ replaced by $\ri\Delta$.
\end{proof}

\section{Disconnectedness and unboundedness}\label{sec_dis}
In this section, we develop and extend the results of Yen and Yao in \cite{YY2011} for general monotone VVIs.

\subsection{Weak Pareto solution sets}

\begin{lemma}\label{A_closed} If $\A$ is a connected component of $\Sol^w(F,K)$, then $\A$ is closed.
\end{lemma}
\begin{proof} To prove the closedness of $\A$, let $\{x^k\}$ be a convergent sequence of points in $\A$ with  $x^k\to\bar x\in K$, we will show that $\bar x\in \A$. According to Theorem \ref{thm_scalarization},  there exists a sequence $\{\xi^{k}\} \subset \Delta$ such that $x^k\in\Sol(F_{\xi^k},K)$ for every $k\in\N$. By the compactness of $\Delta$, we can assume that  $\xi^{k} \to \bar\xi$ and $ \bar\xi\in \Delta$. So, for any $y\in K$, one has	\begin{equation}\label{Fxik}
	\left\langle  F_{\xi^k}(x^{k}),y-x^{k}\right\rangle\geq 0.
	\end{equation}
From \eqref{Fxik}, letting $k\to \infty$, because of $F_{\xi^k}(x^k)\to F_{\bar \xi}(\bar x)$, one has
	$\left\langle F_{\bar\xi}(\bar x),y-\bar x\right\rangle  \geq 0.$ Therefore, $\bar x$ is an weak Pareto solution.
	Since $\A$ is connected, the union $\A\cup\{\bar x\}$ is connected. The maximal connectedness of $\A$  in $\Sol^w(F,K)$ implies that $\A\cup\{\bar x\}\subset \A$, hence that  $\bar x\in\A$. The assertion has been proved.
\end{proof}

\begin{theorem}\label{disconnected_w} Assume that $\VVI(F,K)$ is monotone.  If the weak Pareto solution set is disconnected, then each connected component of this set is unbounded.
\end{theorem}
\begin{proof} Suppose that $\Sol^w(F,K)$ is disconnected. On the contrary, suppose that $\A$ is a bounded connected component of this solution set.  According to Lemma \ref{A_closed}, $\A$ is compact. Applying Proposition \ref{A_compact}, we claim that $S^{-1}(\A)$ is closed. Besides, Proposition \ref{A_open} asserts that $S^{-1}(\A)$ is open in $\Delta$. Hence, $S^{-1}(\A)$ not only is closed but also is open in the induced topology of $\Delta$. Due to the connectedness of $\Delta$, we obtain $S^{-1}(\A)=\Delta$.

For any different connected component $\B$ of $\Sol^w(F,K)$, from Proposition \ref{AB_disjoint}, one has $S^{-1}(\A)\cap S^{-1}(\B)= \emptyset.$	This leads to the contradiction
	$$\emptyset=\Delta\setminus S^{-1}(\A)\supset S^{-1}(\B)\neq \emptyset.$$
	So, $\A$ must be unbounded, and the proof is complete.
\end{proof}

\begin{theorem}\label{connected_w} Assume that $\VVI(F,K)$ is monotone. If the weak Pareto solution set  is bounded and nonempty, then the following assertions are valid:
	\begin{description}
		\item[\rm(a)] The weak Pareto solution set is connected;
		\item[\rm(b)] The domain of the basic multifunction coincides with $\Delta$.
	\end{description}
\end{theorem}
\begin{proof}
	The assertion $\rm(a)$ follows immediately from Theorem \ref{disconnected_w}. The task is now to prove $\rm(b)$.  Since $\Sol^w(F,K)$ is connected and compact, according to Proposition \ref{A_compact} and Proposition \ref{A_open}, the inverse image $S^{-1}(\Sol^w(F,K))$ is closed and open in $\Delta$. Hence, one has $S^{-1}(\Sol^w(F,K))=\Delta.$ The second assertion has been proved.
\end{proof}
\begin{remark}
	In \cite{LY2000}, Lee and Yen have proved that if the constraint set $K$ is bounded then the weak Pareto solution set is connected. This is a special case of the assertion $\rm(a)$ in Theorem \ref{connected_w}.
\end{remark}

\subsection{Proper Pareto solution sets}

\begin{theorem}\label{disconnected_pr} Assume that $\VVI(F,K)$ is monotone. The following assertions are valid:
	\begin{description}
		\item[\rm(a)] If the proper Pareto solution set is disconnected, then each connected component of this set is unbounded;
		\item[\rm(b)] If the proper Pareto solution set  is bounded and nonempty, then this set is connected.
	\end{description}
\end{theorem}
\begin{proof}  Since the second assertion follows obviously the first assertion, we need only to prove $\rm(a)$. Suppose that $\A$ is a connected component of $\Sol^{pr}(F,K)$. From \eqref{Sol_pr} and \eqref{S_1}, it is easy to check that
	\begin{equation}\label{S_12}
	S^{-1}(\A)=S^{-1}(\A)\cap\ri\Delta.
	\end{equation}

Suppose $\Sol^{pr}(F,K)$ is disconnected and, on the contrary, $\A$ is bounded. According to Proposition \ref{A_open}, $S^{-1}(\A)$ is open in $\ri\Delta$.
By the boundedness of $\A$, the closure $\cl(\A)$ is compact and $\cl(\A)\subset \Sol^w(F,K)$. Proposition~\ref{A_compact} asserts that $S^{-1}(\cl(\A))$ is closed.

We claim that
	\begin{equation}\label{cl_A}
	S^{-1}(\A)=S^{-1}(\cl(\A))\cap \ri\Delta.
	\end{equation}
	Indeed, because of $\A\subset\cl(\A)$ one has  $S^{-1}(\A)\subset S^{-1}(\cl(\A))$. From the last result and \eqref{S_12}, one has
	$$S^{-1}(\A) \; \subset \; S^{-1}(\cl(\A))\cap \ri\Delta.$$
	To prove the inverse conclusion, we take $\xi$ in the right-hand side of \eqref{cl_A}, it follows that $\xi\in \ri\Delta$ and $S(\xi)\cap\cl(\A) $ is nonempty. Because $\A$ is a connected component, $S(\xi)$ must be a subset of $\A$. This implies that $\xi\in S^{-1}(\A)$, and hence \eqref{cl_A} is obtained. From \eqref{cl_A}, by the closedness of $S^{-1}(\cl(\A))$, $S^{-1}(\A)$ is closed in $\ri\Delta$.

	Combining these results, $S^{-1}(\A)$ not only is closed but also is open in $\ri\Delta$. Because of the connectedness of $\ri\Delta$, one has  $S^{-1}(\A)=\ri\Delta$.

	Assume that $\B$ is a different connected components of $\Sol^{pr}(F,K)$. Proposition \ref{AB_disjoint} yields
	$S^{-1}(\A)\cap S^{-1}(\B)= \emptyset.$
	This leads to the contradiction
	$$\emptyset=\ri\Delta\setminus S^{-1}(\A)\supset S^{-1}(\B)\neq \emptyset.$$
	So, $\A$ must be unbounded, and the proof is complete.
\end{proof}

\begin{theorem}\label{disconnected} Assume that $K$ is polyhedral convex and $\VVI(F,K)$ is monotone. The following assertions are valid:
	\begin{description}
		\item[\rm(a)] If the Pareto solution set is disconnected, then each connected component of this set is unbounded.
		\item[\rm(b)] If the Pareto solution set is bounded and nonempty, then this set is connected.
\end{description} \end{theorem}
\begin{proof}  Since $K$ is polyhedral convex, according to Theorem \ref{thm_scalarization}, the Pareto solution set coincides with the proper Pareto solution set, i.e., $\Sol(F,K)=\Sol^{pr}(F,K)$. Therefore, the two assertions are implied from Theorem \ref{disconnected_pr}.
\end{proof}

\section{Examples and open questions}\label{sec_question}
The problem $\VVI(F,K)$ is a \textit{symmetric} vector variational inequality if, for each $l\in\{1,\dots,m\}$, the vector-function $F_l$ satisfies the following conditions:
$$
\frac{\partial F_{li}(x)}{\partial x_j}=\frac{\partial F_{lj}(x)}{\partial x_i}\quad\;\forall x\in K,\ \forall i, j\in \{1,\dots,n\},
$$
where $F_{li}(x)$ is the $i$-th component of $F_l(x)$. Besides,
the problem is a \textit{skew-symmetric} vector variational inequality if the following property is available for each $l\in\{1,\dots,m\}$:
$$
\frac{\partial F_{li}(x)}{\partial x_j}=-\frac{\partial F_{lj}(x)}{\partial x_i}\quad\;\forall x\in K,\ \forall i, j\in \{1,\dots,n\}.$$

Note that linear fractional vector optimization
problems and convex quadratic vector optimization problems which are two important classes of vector optimization problems can be treated as skew-symmetric affine variational inequalities and symmetric monotone affine variational inequalities (see, e.g., \cite{YP2000,YY2011}).

\subsection{Examples}

\begin{example}	Consider the unconstrained bicriteria variational inequality problem $(Q)$ which is defined by $K=\R^2$ and \begin{equation}\label{Q_F12}
	F_1(x)=\begin{bmatrix}
	x_1^3\\
	x_2^3-1
	\end{bmatrix}, \  F_2(x)=\begin{bmatrix}
	x_1^3-1\\
	x_2^3
	\end{bmatrix}.
	\end{equation}
	Clearly,  the Jacobian matrices of $F_1$ and $F_2$ are symmetric positive semidefinite
	$$DF_1(x)=DF_2(x)=\begin{bmatrix}
	3x_1^2&0\\
	0&3x_2^2
	\end{bmatrix}.$$
	Hence, $F_1, F_2$ are monotone on $\R^2$. For each $\xi\in\Delta$, where
	$$\Delta=\{(\xi_1,1-\xi_1)\in\R^2:\xi_1\in[0,1]\},$$
	$x$ is a solution of $\VI(F_{\xi},\R^2)$ if and only if $x$ is a solution of the equation
	$$\xi_1F_1(x)+(1-\xi_1)F_2(x)=0.$$
	From \eqref{Q_F12}, one has
	\begin{equation*}
	\begin{bmatrix}
	x_1^3\\
	x_2^3
	\end{bmatrix}=\begin{bmatrix}
	1-\xi_1\\
	\xi_1
	\end{bmatrix}.
	\end{equation*}
	By the relation \eqref{scalarization}, the weak Pareto solution set is given by
	$$\Sol^w(\Q)=\left\lbrace \left( \sqrt[3]{1-\xi_1},\sqrt[3]{\xi_1}\right) :\xi_1\in \left[ 0,1\right]   \right\rbrace.$$
It is easy to see that $\Sol^w(\Q)$ is bounded. Theorem \ref{connected_w} asserts that this set is connected and the domain of the basic multifunction coincides with $\Delta$.

	\begin{center}
		\begin{tikzpicture}
		\begin{axis}[axis lines = center,
		xlabel = $x_1$,
		ylabel = $x_2$,]
		\addplot [very thick, domain=0:1,
		samples=200,
		color=black,]{(1-x^3)^(1/3)};
		\addlegendentry{$\Sol^w(\Q)$}
		\end{axis}
		\end{tikzpicture}

		{\textbf{Figure 1} \ The weak Pareto solution set $\Sol^w(\Q)$}
	\end{center}

	Since $K=\R^2$, according to Theorem \ref{thm_scalarization}, the Pareto solution set coincides with the proper Pareto solution set. This means that
	$$\Sol(\Q)=\Sol^{pr}(\Q)=\left\lbrace \left( \sqrt[3]{1-\xi_1},\sqrt[3]{\xi_1}\right) :\xi_1\in \left( 0,1\right)   \right\rbrace.$$
	The set $\Sol(\Q)$ is bounded. By Theorem \ref{disconnected}, this one is connected.
\end{example}

\begin{example}\label{example2}
	Consider the unconstrained bicriteria variational inequality problem $(\Po)$ which is defined by $K=\R^2$ and
	\begin{equation}\label{F12}
	F_1(x)=\begin{bmatrix}
	-x_2-1\\
	x_1^3-1
	\end{bmatrix}, \ F_2(x)=\begin{bmatrix}
	x_2-1\\
	-x_1^3-1
	\end{bmatrix}.
	\end{equation}
	The Jacobian matrices of $F_1$ and $F_2$ are non-symmetric positive semidefinite
	$$DF_1(x)=\begin{bmatrix}
	0&-1\\
	3x_1^2&0
	\end{bmatrix}, \ DF_2(x)=\begin{bmatrix}
	0&1\\
	-3x_1^2&0
	\end{bmatrix}.$$
	Hence, each of $F_1, F_2$ is monotone on $\R^2$. For each $\xi\in\Delta$, $x$ is a solution of $\VI(F_{\xi},\R^2)$ if and only if $x$ is a solution of the equation
	$$\xi_1F_1(x)+(1-\xi_1)F_2(x)=0.$$
	From \eqref{F12}, one has
	$$\begin{bmatrix} 	(1-2\xi_1)x_2\\
	(2\xi_1-1)x_1^3
	\end{bmatrix}=\begin{bmatrix} 1 \\
	1
	\end{bmatrix}.$$
Clearly, the basic multifunction of $(\Po)$ is given by
	$$S(\xi_1,1-\xi_1)=\left\{\begin{array}{cl}
	\Big\{\big(\frac{1}{\sqrt[3]{2\xi_1-1}},\frac{-1}{{2\xi_1-1}}\big)\Big\} & \quad \hbox{ if } \ \ \xi_1\in \big[ 0,\frac{1}{2}\big) \cup \big( \frac{1}{2},1\big],  \\
	\emptyset & \quad \hbox{ if } \ \ \xi_1=\frac{1}{2}.
	\end{array}\right.$$
	So, the weak Pareto solution set is determinated as
	$$\Sol^w(\Po)=\left\lbrace 	\big(\frac{1}{\sqrt[3]{2\xi_1-1}},\frac{-1}{{2\xi_1-1}}\big) :\xi_1\in \big[ 0,\frac{1}{2}\big) \cup \big(\frac{1}{2},1\big]   \right\rbrace.$$
	Clearly, the weak Pareto solution set has two connected components. According to Theorem \ref{disconnected_w}, each component of this set is unbounded.

	\begin{center}
		\begin{tikzpicture}
		\begin{axis}[axis lines = center, 	xlabel = $x_1$,	ylabel = $x_2$,]
		\addplot [thin, dashed, domain=-1.5:1.5, samples=100, color=gray]{-x^3};
		\addplot [very thick, domain=-1.5:-1, 		samples=1000, 		color=black]{-x^3};
		\addplot [very thick, domain=1:1.5, 		samples=1000, 		color=black]{-x^3};
		\addlegendentry{$x_2 =-x^3_1$}
		\addlegendentry{$\Sol^w(\Po)$}
		\end{axis}
		\end{tikzpicture}

		{\textbf{Figure 2} \ The weak Pareto solution set $\Sol^w(\Po)$}
	\end{center}

	Because the constraint set is $\R^n$,  $\Sol(\Po)$ coincides with $\Sol^{pr}(\Po)$. Hence,
	the Pareto solution set $\Sol(\Po)$ is given by
	$$\Sol(\Po)=\Big\lbrace 	\big(\frac{1}{\sqrt[3]{2\xi_1-1}},\frac{-1}{{2\xi_1-1}}\big) :\xi_1\in \big(0,\frac{1}{2}\big) \cup \big(\frac{1}{2},1\big)     \Big\rbrace.$$
	This one has two unbounded connected components.
\end{example}

\subsection{Open questions}

We conclude this section by two open questions.

\begin{question} Is it true that if the Pareto solution set of a monotone VVI is disconnected, then  this set is unbounded, or not?
\end{question}

\begin{question} Is it true that if the weak Pareto solution set (or, the Pareto solution set, the proper Pareto solution set) of a symmetric monotone VVI is nonempty, then this set is connected, or not?
\end{question}

\section*{Acknowledgements}
The author would like to thank Professor Nguyen Dong Yen for encouragement and the anonymous referees for valuable remarks and suggestions.


\section*{Funding}
This research was supported by Vietnam National Foundation for Science and Technology Development (NAFOSTED) [grant number 101.01--2018.306].

\end{document}